\documentclass[11pt, letterpaper]{amsart}
\usepackage[active]{srcltx}
\usepackage{calc,amssymb,amsthm,amsmath, ulem}
\usepackage{alltt}
\RequirePackage[dvipsnames,usenames]{color}

\input{kmacros3.sty}
\input{xy}
\xyoption{all}

\DeclareMathOperator{\Id}{{Id}}
\newcommand{\F}{\mathbb{F}}
\renewcommand{\:}{\colon}

\newcommand{\eg}{{\itshape e.g.} }
\newcommand{\ie}{{\itshape i.e.} }

\newcommand{\bp}{\mathfrak{p}}
\newcommand{\bq}{\mathfrak{q}}
\newcommand{\Ram}{\mathrm{Ram}}

\usepackage[left=1in,top=1.0in,right=1in,bottom=1.0in]{geometry}

\usepackage{setspace}
\usepackage{hyperref}
\usepackage{enumerate}
\usepackage{graphicx}

\usepackage[all,cmtip]{xy}
\usepackage{verbatim}

\renewcommand{\O}{\mathcal O}

\begin{document}

\title[Extending Frobenius Splittings]{Explicitly Extending Frobenius Splittings over Finite Maps}
\author{Karl Schwede and Kevin Tucker}

\address{Department of Mathematics\\ The Pennsylvania State University\\ University Park, PA, 16802, USA}
\email{schwede@math.psu.edu}
\address{Department of Mathematics\\ Princeton University\\Princeton, NJ, 08542, USA}
\email{kftucker@math.princeton.edu}

\thanks{The first author was partially supported by a National Science Foundation postdoctoral fellowship, RTG grant number 0502170 and NSF DMS 1064485/0969145.}
\thanks{The second author was partially supported by RTG grant number 0502170 and a National Science Foundation postdoctoral fellowship DMS 1004344.}

\begin{abstract}
Suppose that $\pi \: Y \to X$ is a finite map of normal varieties over a perfect field of characteristic $p > 0$.  Previous work of the authors gave a criterion for when Frobenius splittings on $X$ (or more generally any $p^{-e}$-linear map) extend to $Y$.  In this paper we give an alternate and highly explicit proof of this criterion (checking term by term) when $\pi$ is tamely ramified in codimension 1.  Some additional examples are also explored.
\end{abstract}
\maketitle
%\tableofcontents
\numberwithin{equation}{theorem}

\section{Introduction}
\label{sec:Introduction}

In \cite{SchwedeTuckerTestIdealFiniteMaps}, the authors describe transformation behavior for Frobenius splittings and test ideals under finite covers $\pi \: Y \to X$.  The main tool was the Grothendieck trace map, which was used to analyze extension properties of Frobenius splittings or more generally other so-called $p^{-e}$-linear maps.  This method, while necessary to account for arbitrary covers, has the disadvantage that it is rather opaque in that it is hard to see precisely how individual splittings extend from $X$ to $Y$.  However, in the case that the finite cover $\pi$ has tame ramification in codimension one, it is possible to give a direct, elementary, and explicit term by term analysis of the situation.

In this paper, we give an alternative elementary self-contained proof of the following theorem, first shown in \cite{SchwedeTuckerTestIdealFiniteMaps} (\textit{cf.} Proposition 4.4 and Theorem 5.7 in \emph{op. cit.}).

\begin{mainthm}
Suppose $\pi \: Y \to X$ is a finite surjective morphism of normal $F$-finite algebraic varieties.  Assume that $\pi$ is both separable and tamely ramified in codimension one.  Then a homomorphism $\psi \in \Hom_{\O_{X}}(F^{e}_{*}\O_{X}, \O_{X})$ extends to a homomorphism $\bar{\psi} \in \Hom_{\O_{Y}}(F^{e}_{*}\O_{Y}, \O_{Y})$ if and only if $\pi^* \Delta_{\psi} \geq \Ram_{\pi}$.  Furthermore in this case, we have $\Delta_{\bar{\psi}} = \pi^* \Delta_{\psi} - \Ram_{\pi}$.
\end{mainthm}

\noindent
Here, $\Ram_{\pi}$ is the ramification divisor of $\pi$ and $\Delta_{\psi}$ is the effective $\Q$-divisor associated to $\psi \in \Hom_{\O_{X}}(F^{e}_{*}\O_{X}, \O_{X})$ as in \cite[Section~2.5]{SchwedeTuckerTestIdealFiniteMaps} or as briefly described below.  Furthermore, in the context of our main theorem, it is easy to see that $\psi$ is a Frobenius splitting if and only if $\bar\psi$ is.

As mentioned, this result is not as general as the results from \cite{SchwedeTuckerTestIdealFiniteMaps}.  On the other hand, we feel that the proof we give here is rather valuable as it renders transparent the explicit obstruction to extending a $p^{-e}$-linear map.  Notably, the proof in this paper originally appeared as an appendix to the arXiv version of \emph{op. cit.} as they in fact \emph{preceded} and led to those more general results.

\vskip 12pt
\hskip -12pt{\it Acknowledgements:}  The authors would like to thank
Karen Smith for several inspiring discussions.

\section{Background and reduction to dimension 1}
\label{sec:Background}

To prove the main theorem, it suffices to assume that $X = \Spec(R)$ and $Y = \Spec(S)$ are affine.  Thus, consider $R \subseteq S$ a module-finite inclusion of $F$-finite
normal domains of finite type over an $F$-finite field of characteristic $p>0$. If $F \: X \to X$ is the (absolute) Frobenius endomorphism, recall that the map of rings $\O_{X} \to F_{*} \O_{X}$ is then naturally identified with the inclusion $R \subseteq R^{1/p}$ of $R$ into its ring of $p$-th roots $R^{1/p}$ (inside a fixed algebraic closure of its fraction field); similarly, the inclusion $R \subseteq R^{1/p^{e}}$ is identified with the $e$-iterated Frobenius.
Given the separability assumption on $\pi$, we have that the corresponding
inclusion of fields $K := \Frac(R) \subseteq L := \Frac(S)$
is separable.

  % Assume further that $R$ and $S$ satisfy (\ref{eqn:FUpperShriekOmegaIsOmega}); at the risk of needless repetition, we remark that property (\ref{eqn:FUpperShriekOmegaIsOmega}) is automatically satisfied when $R$ and $S$ are semi-local or are essentially of finite type over a perfect or even $F$-finite field.
%things that (still?) need referenced are: bravo/smith, hara/takagi, watanabe (blah ... normal graded ...), and hoch/hun(blah briancon-skoda, section 6.3)
We will give a criterion for when a $p^{-e}$-linear
map $\phi
\in \Hom_{R}(R^{1/p^e},R)$ on $R$ extends to a $p^{-e}$-linear
map \mbox{$\bar{\phi}
\in \Hom_{S}(S^{1/p^e}, S)$} on $S$, \ie such that the following diagram
commutes:
\begin{center}
\begin{minipage}[b]{\linewidth}
\[  \xymatrix{
S^{1/p^e} \ar^{\bar{\phi}}[r] & S
}\]

\vspace{-.5cm}
\begin{equation}
\label{eq:extndiag}
 \xymatrix{
 \rotatebox[origin=c]{90}{$\subseteq$} & \quad \rotatebox[origin=c]{90}{$\subseteq$}
 }\end{equation}

\vspace{-.5cm}
\[ \xymatrix{
R^{1/p^e} \ar_{\phi}[r] & R
}. \]
\end{minipage}
\end{center}

Let us give a quick example.

\begin{example}
\label{ex:y=x4}
  Suppose $R = \F_{3}[y] \subseteq \F_{3}[x,y]/(y - x^{4}) \simeq
  \F_{3}[x] = S$.  Then $R^{1/3}$ is a free \mbox{$R$-module} with basis $1,
  y^{1/3}, y^{2/3}$, and similarly for $S^{1/3}$.  Thus, $\phi \in
  \Hom_{R}(R^{1/3},R)$ is uniquely determined by $\phi(1),
  \phi(y^{1/3}), \phi(y^{2/3})$, which can be arbitrary elements of
  $R$.  We identify $\phi$ with its generic extension $\phi \: K^{1/3}
  \to K$ to $K = \Frac(R)$, and denote by $\bar\phi$ the unique extension of $\phi$ to
  $L = \Frac(S)$ as in Lemma \ref{lem:genericextension} below.  We have
\[
\begin{array}{rcl}
    \bar\phi(1) & = & \phi(1) \\
    \bar\phi(x^{1/3}) & = & \bar\phi({1 \over x} x^{4/3}) = \frac{1}{x}
    \phi(y^{1/3}) \\
    \bar\phi(x^{2/3}) & = & \bar\phi({1 \over x^2} x^{8/3}) = \frac{1}{x^2} \phi(y^{2/3})
\end{array}
\]
and it follows that an extension of $\phi$ from $R$ to $S$ as in
\eqref{eq:extndiag} exists if and only if
\begin{itemize}
\item{} $\phi(y^{1/3})$ is divisible by $x$ in $S$ and,
\item{}  $\phi(y^{2/3})$ is divisible by $x^2$ in $S$.
\end{itemize}
Note that $\phi(y^{1/3})$ is divisible by $x = y^{1/4}$ in $S$ if and only if $\phi(y^{1/3})$ is divisible by $y = x^4$ in $R$.  Likewise $\phi(y^{2/3})$ is divisible by $x^2 = y^{2/4}$ in $S$ if and only if $\phi(y^{2/3})$ is divisible by $y = x^4$ in $R$.
\end{example}

Our eventual goal is to generalize the work in the above example to any inclusion of $F$-finite normal domains and in particular to obtain an explicit criterion involving the ramification divisor.  In particular, our final description will be given in terms of the divisor associated to
$\Delta_{\phi}$ on $X = \Spec R$ as in \cite[Section 2.5]{SchwedeTuckerTestIdealFiniteMaps}.  We now briefly recall the construction of this divisor.

The $F^e_* \O_X$-module $\sHom_{\O_X}(F^e_* \O_X, \O_X)$ is isomorphic to $F^e_* \O_X((1-p^e)K_X)$ and so a section $\phi \in \sHom_{\O_X}(F^e_* \O_X, \O_X)$ corresponds to an effective Weil divisor $D_{\phi}$ linearly equivalent to $(1-p^e)K_X$.  This divisor is ubiquitous throughout the theory of Frobenius splitting \cite{BrionKumarFrobeniusSplitting}.  We set
\[
\Delta_{\phi} = {1 \over p^e - 1} D_{\phi}.
\]
The division by $p^e - 1$ normalizes $D_{\phi}$ with respect to $e$ and $p$ and also forces $K_X + \Delta$ to be $\bQ$-Cartier, a common assumption for $\bQ$-divisors throughout higher dimensional algebraic geometry.

Our goal for the remainder of this section is to reduce the question of our main theorem to the case where $R$ and $S$ are one dimensional.  First we recall how to extend $\phi$ in the case where $R$ and $S$ are fields.  The proof is easy and so we include it for the convenience of the reader.  We remark that if the extension of fields is \emph{not} separable, then no non-zero $p^{-e}$-linear map extends by \cite[Proposition 5.2]{SchwedeTuckerTestIdealFiniteMaps}.

\begin{lemma}\cite[Lemma 3.3]{SchwedeTuckerTestIdealFiniteMaps}
\label{lem:genericextension}
Suppose that $K \subseteq L$ is a finite separable field extension
  of $F$-finite fields with characteristic $p>0$.  If
\[
 e_{1}^{1/p^e}, \ldots, e_{n}^{1/p^e}
 \]
 form a basis for $K^{1/p^e}$ over $K$, then they also form a basis for $L^{1/p^e}$ over $L$. Furthermore, any
  map $\phi \in \Hom_{K}(K^{1/p^e},K)$ extends uniquely
  to a map $\bar{\phi} \in \Hom_{L}(L^{1/p^e}, L)$.
\end{lemma}

\begin{proof}
 Since $K^{1/p^e}$ is purely inseparable over $K$, it follows that
 $K^{1/p^e}$ and $L$ are linearly disjoint over $K$ (\eg Example 20.13
 in \cite{Mor96}).  Thus, $L^{1/p^e} = K^{1/p^e} \tensor_{K}
 L$, and the first statement is obvious.  For the second, note that
 the extension $\bar\phi = \phi \tensor_{K} \Id_{L}$ exists and is
 uniquely determined by the property that $\bar\phi(e_{i}^{1/p^e}) =
 \phi(e_{i}^{1/p^e})$ for all $1 \leq i \leq n$.
\end{proof}

We will also need the following lemma.

\begin{lemma} \cite[Lemma 4.7]{SchwedeTuckerTestIdealFiniteMaps}
\label{lem:CharacterizationOfGeneratingMap}
Let $R \subseteq S$ be a module-finite inclusion of domains with
corresponding fraction fields $K \subseteq L$.
Suppose that $\Phi \in \Hom_{K}(L,K)$ satisfies $\Phi(S) \subseteq R$, and
that $\Phi|_{S} \: S \to R$ generates $\Hom_R(S, R)$ as an $S$-module.
If $x \in L$ is such that $\Phi(x S) \subseteq R$, then $x \in S$.
\end{lemma}
\begin{proof}
The map $\phi(\blank) := \Phi(x \cdot \blank)$ can be viewed as an
element of $\Hom_R(S, R)$.  Therefore $\phi(a) = \Phi(s a)$ for some
$s$ in $S$ and all $a \in S$.  But then $\Phi(s l) = \Phi(x l)$ for
all $l \in L$, which implies that $s = x$ as $\Hom_K(L, K)$ is a
one dimensional vector space over $L$.
\end{proof}

Finally as promised, we use localization to reduce our extension question to
codimension one.% \ie to the case of a discrete valuation ring (DVR).

\begin{lemma} \cite[Lemma 3.6]{SchwedeTuckerTestIdealFiniteMaps}
\label{lem:reductcodimoneaffine}
  Suppose $R \subseteq S$ is a generically separable module-finite inclusion
  of $F$-finite normal domains with characteristic $p>0$ and $\phi \in
  \Hom_{R}(R^{1/p^e},R)$.  Then $\phi$ extends to
  $\bar\phi \in \Hom_{S}(S^{1/p^e},S)$ if and only if an extension exists
  in codimension one.  In other words, for each height one prime $\bq$
  of $S$ lying over $\bp$ in $R$, we have a commutative diagram
\begin{center}
\begin{minipage}{\linewidth}
\[  \xymatrix{
(S_{\bq})^{1/p^e} \ar^-{\bar{\phi}}[r] & S_{\bq}
}\]

\vspace{-.5cm}
\begin{equation*}
 \xymatrix{
 \rotatebox[origin=c]{90}{$\subseteq$} & \quad \quad \rotatebox[origin=c]{90}{$\subseteq$}
 }\end{equation*}

\vspace{-.5cm}
\[ \xymatrix{
(R_{\bp})^{1/p^e} \ar_-{\phi}[r] & R_{\bp}
}. \]
\end{minipage}
\end{center}
\end{lemma}

\begin{proof}
  %Since $Y = \Spec(S)$ is normal, the above statements follow immediately from the  fact that $\Hom_{\O_{Y}}(F^{e}_{*}\O_{Y},\O_{Y})$ is reflexive and  thus its sections are determined outside of a codimension two subset of $\Spec(S)$  (see for example \cite[Proposition  1.11]{HartshorneGeneralizedDivisorsOnGorensteinSchemes}).
%We give an  alternative argument which we find instructive.
Identify $\phi$ with its generic extension $\phi \: K^{1/p^e} \to   K$ to $K = \Frac(R)$, and denote by $\bar\phi$ the unique extension  of $\phi$ to $L = \Frac(S)$ as in Lemma \ref{lem:genericextension}.  Then, an extension of $\phi$ to $S$ exists if and only if  $\bar\phi(S^{1/p^e}) \subseteq S$.  Since $S$ is normal,  $S$ is the intersection of all of the subrings $S_{\bq} \subseteq L$  for each height one prime $\bq$ of $S$, and the conclusion follows at once.
\end{proof}

\section{On explicitly extending $p^{-e}$ linear maps under tame ramification}
\label{sec:explicitlifting}

Our goal in this section is to give a highly explicit description of
how to extend $p^{-e}$-linear maps in separable extensions of normal domains which
are tamely ramified in codimension one.  In doing so, we will prove our main theorem.  We first recall the definition of tame ramification:

\begin{definition}  \cite[Chapter 2]{GrothendieckMurreTheTameFundamentalGroup} \label{defn:tameramification}
 A local inclusion $(R, \bp) \subseteq (S, \bq)$ of DVR's is called \emph{tamely ramified} if:
\begin{itemize}
 \item[(1)]  It is generically finite and generically separable.
 \item[(2)]  The extension of residue fields $k(\bp) \subseteq k(\bq)$ is separable.
 \item[(3)]  A local parameter $r$ of $R$, when viewed as an element of $S$, has order of vanishing (with respect to the valuation of $S$) not divisible by $p$.
\end{itemize}
More generally, a generically separable module-finite inclusion of
normal domains $R \subseteq S$ will be called \emph{tamely ramified in
  codimension one} if all of the associated DVR extensions given by
localizing at height one primes are tamely ramified.  A generically
finite and separable inclusion of DVR's which is not tamely ramified is called \emph{wildly ramified}.  In this context, the \emph{ramification index} is the order of vanishing of $r$ in $S$ (with respect to the valuation of $S_{\bq}$).
Finally, if $R \subseteq S$ is tamely ramified with index $n$ and $s$ is a local parameter of $S$, then the ramification divisor $\Ram_{\pi}$ on $Y = \Spec S$ is simply $(n-1) \Div(s)$.
\end{definition}

By Section~\ref{sec:Background}, the problem reduces to codimension one,
and so we will need only consider a tamely ramified generically finite extension of  DVR's $R \subseteq S$.
Before proceeding, we would like to briefly outline the strategy.

\begin{strategy*}
For particular $e$, we will be able to find a map $\phi : R^{1/p^e} \to R$ whose extension to $S$ generates $\Hom_S(S^{1/p^e}, S)$ as an $S$-module in Proposition \ref{prop:explictMapLiftsToGenerator}.  Those $\psi \in \Hom_R(R^{1/p^e}, R)$ which extend to $S$ will be exactly those of the form $\psi(\blank) = \phi(c^{1/p^e} \cdot \blank)$ for some $c^{1/p^e} \in R^{1/p^e}$.  In particular, this $\phi$ will create a \emph{line-in-the-sand} that determines which $\psi$ extend to $S$.  The proof of our main theorem is then completed by using Lemma \ref{lem.CompositionExtendsIfAndOnlyIf} to reduce the general case to these particular values for $e$.
\end{strategy*}

Our first order of business is to explicitly write down a basis for $R^{1/p^e}$ over $R$.

\begin{lemma}
\cite[Lemma 3.6]{SchwedeTuckerTestIdealFiniteMaps}
\label{lem:bases}
Let $R$ be a DVR with parameter $r$, characteristic $p>0$, and $F$-finite residue field $k$.  For each $i = 0, \ldots, p^e-1$, suppose that $u_{1,i}^{1/p^e}, \ldots, u_{m,i}^{1/p^e}$ are elements of $R^{1/p^e}$ whose images $\bar{u}_{j,i}^{1/p^e} \in k^{1/p^e} = R^{1/p^e}/ r^{1/p^e}R^{1/p^e}$ form a basis for $k^{1/p^e}$ over $k$.  Then
\[
\begin{array}{c}
u_{1,0}^{1/p^e}, \ldots, u_{m,0}^{1/p^e}, \\[.05in]
u_{1,1}^{1/p^e}r^{1/p^e}, \ldots, u_{m,1}^{1/p^e}r^{1/p^e}, \\
\vdots \\
u_{1,i}^{1/p^e}r^{i/p^e}, \ldots, u_{m,i}^{1/p^e}r^{i/p^e},\\
\vdots \\
u_{1,p^e-1}^{1/p^e}r^{(p^e-1)/p^e}, \ldots, u_{m,p^e-1}^{1/p^e}r^{(p^e-1)/p^e},\\
\end{array}
\]
are a free basis for $R^{1/p^e}$ over $R$.
\end{lemma}

\begin{proof}
 Since $R$ is regular, we have from \cite{KunzOnNoetherianRingsOfCharP} that $R^{1/p^e}$ is a free $R$-module.  Thus, by Nakayama's~lemma, it is enough to show that the images of the elements in the above list form a basis for the vector space $R^{1/p^e}/rR^{1/p^e}$ over $k$.  Consider the filtration
 \[
 rR^{1/p^e} \subsetneq r^{(p^e-1)/p^e}R^{1/p^e} \subsetneq \cdots \subsetneq r^{i/p^e}R^{1/p^e} \subsetneq \cdots \subsetneq r^{1/p^e}R^{1/p^e} \subsetneq R^{1/p^e}.
 \]
 Multiplication by $r^{i/p^e}$ induces an isomorphism
 \[
 k^{1/p^e} = R^{1/p^e}/rR^{1/p^e} \cong r^{i/p^e}R^{1/p^e} \big/ r^{(i+1)/p^e}R^{1/p^e}
 \]
 for each $i = 0, \ldots, p^e-1$.  It follows immediately that
 \[
 r^{i/p^e}R^{1/p^e} = u_{1,i}^{1/p^e}r^{i/p^e}R + \cdots + u_{m,i}^{1/p^e}r^{i/p^e}R + r^{(i+1)/p^e}R^{1/p^e}
 \]
 and, arguing inductively, it is clear that the images of the elements in the above list span $R^{1/p^e}/rR^{1/p^e}$.  Additionally, we can conclude from the above filtration that the dimension of $R^{1/p^e}/rR^{1/p^e}$ over $k$ equals $p^e \cdot [k^{1/p^e}:k]$.   Since there are precisely $p^e \cdot [k^{1/p^e}:k]$ elements in the above list, the conclusion now follows.
 \end{proof}

Once we have fixed a basis of $R^{1/p^e}$ over $R$ (a DVR) we can very explicitly write down an $R^{1/p^e}$-module generator of $\Hom_R(R^{1/p^e}, R)$.

\begin{lemma}
\label{lem:GenMapDescription}
Fix an $F$-finite DVR $R$ and suppose that $u_{j, i}^{1/p^e} r^{i/p^e}$ where $i = 0, \dots, p^e-1$, $j = 1, \dots, m$ is a basis for $R^{1/p^e}$ over $R$ as described in Lemma \ref{lem:bases}.  Fix one of the basis elements $x^{1/p^e} = (u_{j, p^e-1} r^{p^e-1})^{1/p^e}$ and suppose that we have a map $\phi_x \: R^{1/p^e} \to R$ which sends $x^{1/p^e}$ to $1$.  Then $\phi_x$ generates $\Hom_R(R^{1/p^e}, R)$ as an $R^{1/p^e}$-module.
\end{lemma}
\begin{proof}
Since $u_{j, p^e-1}^{1/p^{e}}$ is a unit in $R^{1/p^{e}}$, it suffices
to show that $\phi (\blank) :=\phi_{x}(u_{j, p^e-1}^{1/p^{e}}
\cdot \blank)$ generates $\Hom_R(R^{1/p^e}, R)$ as an
$R^{1/p^e}$-module.  This statement may be checked after completion, and thus we may
assume that $R =
k[[r]]$.
Fix a basis $1={v_1}^{1/p^e},{v_2}^{1/p^e}, \dots, {v_m}^{1/p^e}$ for
$k^{1/p^e}$ over $k$ (and note these are elements of $R^{1/p^e}$ since
$R$ contains $k$).  Let $\sigma_{j} \in \Hom_{k}(k^{1/p^{e}},k)$
denote the projection onto $v_{j}$ for $j = 1, \ldots, m$.  Since
$\sigma_{j} \in \Hom_{k}(k^{1/p^{e}},k) \simeq k^{1/p^{e}}$ is a
one-dimensional vector space over $k^{1/p^{e}}$, there exists
$\alpha_{j}^{1/p^{e}} \in k^{1/p^{e}}$ with $\sigma_{j}(\blank) =
\sigma_{1}(\alpha_{j}^{1/p^{e}} \cdot \blank)$ for each $j$. As
$\sigma_{1}, \ldots, \sigma_{m}$ are a basis for
$\Hom_{k}(k^{1/p^{e}},k)$ over $k$, it follows that
$\alpha_{1}^{1/p^{e}}, \ldots,
\alpha_{m}^{1/p^{e}}$ are another basis for $k^{1/p^e}$ over $k$.

 Consider
the map $\psi \in \Hom_R(R^{1/p^e}, R)$ which sends
$r^{(p^e-1)/p^e}$ to $1$ and the other elements of the alternate basis
$\{v_j^{1/p^e} r^{i/p^e} \}$ where $j = 1, \dots, m$ and $i = 0, \dots, p^e-1$
to zero.  It is easy to see that $\psi$ generates $\Hom_R(R^{1/p^e},
R)$ as an $R^{1/p^e}$-module.  Indeed,
$\psi(\alpha_{j}^{1/p^{e}}r^{(p^e-1 - i)/p^{e}} \cdot \blank )$ sends
$v_j^{1/p^e} r^{i/p^e}$ to $1$ and the other (alternate) basis
elements to zero.
Now, notice that $\phi(\blank) = \psi(s^{1/p^e} \cdot \blank)$ where
\[
s^{1/p^{e}} = \sum_{i,j} \phi(v_{j}^{1/p^{e}}r^{i/p^{e}})\alpha_{j}^{1/p^{e}} r^{(p^{e}-1-i)/p^{e}}
\]
and we will be done if $s^{1/p^{e}}$ is a unit in
$R^{1/p^{e}}$.  Denoting images in $k^{1/p^e} = R^{1/p^e}/
r^{1/p^e}R^{1/p^e}$ with a bar, we have
\[
\overline{s^{1/p^{e}}} = \sum_{j}
\overline{\phi(v_{j}^{1/p^{e}}r^{i/p^{e}})\alpha_{j}^{1/p^{e}}} = \sum_{j}
\left(\overline{\phi(v_{j}^{1/p^{e}}r^{i/p^{e}})}\right) \alpha_{j}^{1/p^{e}}
\]
which is non-zero as $\alpha_{1}^{1/p^{e}}, \ldots,
\alpha_{m}^{1/p^{e}}$ are linearly independent over $k$ and $\overline{\phi(v_{j}^{1/p^{e}}r^{i/p^{e}})} \in
k$ with $\overline{\phi(v_{1}^{1/p^{e}}r^{i/p^{e}})} =
\overline{\phi(r^{i/p^{e}})} = \bar{1}=1$.
%
\begin{comment}
Set $N \subseteq \Hom_R(R^{1/p^e}, R)$ to be the $R^{1/p^e}$-submodule
generated by $\phi_x$.  We will show that this inclusion is equality.
Tensoring with the completion of $R$ allows us to assume that $R =
k[[r]]$ so that $x^{1/p^e} = (w r^{p^e-1})^{1/p^e}$ for some unit $w
\in R$.  Fix a basis ${v_1}^{1/p^e}, \dots, {v_m}^{1/p^e}$ for
$k^{1/p^e}$ over $k$ (and note these are elements of $R^{1/p^e}$ since
$R$ contains $k$).  Consider the map $\psi$ which sends
$r^{(p^e-1)/p^e}$ to $1$ and the other elements of the alternate basis
$\{v_j^{1/p^e} r^{i/p^e} \}$ where $j = 1, \dots, m$ and $i = 0, \dots, p^e-1$
to zero.  We will show that $\psi$ generates $\Hom_R(R^{1/p^e}, R)$ as an $R^{1/p^e}$-module.  But this is easy, since $\psi( (v_j)^{-1/p^e} r^{(p^e-1 - i)/i} \cdot \blank)$ sends $v_j^{1/p^e} r^{i/p^e}$ to $1$ and the other (alternate) basis elements to zero.

Now, notice that $\phi_x(\blank) = \psi_x(s^{1/p^e} \times \blank)$, so if we can show that $s^{1/p^e}$ is a unit, we will be done.  But,
\[
\phi_x\left((u_{j, p^e-1} r^{p^e-1})^{1/p^e} \right) = 1
\]
so that $\psi_x(u_{j, p^e-1}^{1/p^e} s^{1/p^e} \times \blank) =
\phi_x(u_{j, p^e-1}^{1/p^e} \times \blank)$ sends $r^{(p^e-1)/p^e}$ to
$1$.  It is then straightforward to verify that $u_{j, p^e-1}^{1/p^e}
s^{1/p^e}$ is a unit so that $s^{1/p^e}$ is also a unit as desired.
\end{comment}
\end{proof}

Now suppose we have a (not necessarily finite) local extension of DVR's $R \subseteq S$ which is
generically finite and separable, as well as tamely ramified.  We identify a map $\phi \in \Hom_R(R^{1/p^e}, R)$ which extends to a map $\bar{\phi} \in \Hom_S(S^{1/p^e}, S)$ generating $\Hom_S(S^{1/p^e}, S)$ as an \mbox{$S^{1/p^e}$-module}.

\begin{proposition}
\label{prop:explictMapLiftsToGenerator}
Suppose that $R \subseteq S$ is a generically finite and separable
local extension of DVR's with associated extension of fraction fields
$K \subseteq L$.  Further suppose that the (not necessarily finite)
extension $R \subseteq S$ is tamely ramified.  Then there exists an $e
> 0$ and $\phi \in \Hom_R(R^{1/p^e}, R)$ which extends to an element
$\bar{\phi} \in \Hom_S(S^{1/p^e}, S)$ generating $\Hom_S(S^{1/p^e},
S)$ as an $S^{1/p^e}$-module.  Furthermore, we will explicitly see
that if $\pi \: Y =\Spec S \to X =\Spec R$ is the induced map, then $\pi^*
\Delta_{\phi} = \Ram_{\pi}$.
\end{proposition}
\begin{proof}
Fix $r$ to be a parameter for $R$ and $s$ to be a parameter for $S$
and write $r = u s^n$ for some unit $u \in S$.  Here $n$ is the
ramification index of $R \subseteq S$ and so $p$ does not divide $n$.
Set $e$ to be a number such that $n$ divides $(p^e - 1)$ and fix $b =
(p^e - 1)/n$.  Let $k=R/rR$ and $l = S/sS$, and fix a basis
$\overline{u_1}^{1/p^e}, \dots, \overline{u_m}^{1/p^e}$ for $k^{1/p^e}$ over $k$
(with corresponding lifts to elements $u_i \in R$).

The following elements form a basis for $R^{1/p^e}$ over $R$ by Lemma \ref{lem:bases}.
\begin{equation}
  \label{eq:oldbasis}
\begin{array}{c}
u_1^{1/p^e}, \ldots, u_m^{1/p^e}\\
u_1^{1/p^e} r^{1/p^e}, \ldots, u_m^{1/p^e} r^{1/p^e}\\
\dots\\
u_1^{1/p^e} r^{(p^e - 1)/p^e}, \ldots, u_m^{1/p^e} r^{(p^e - 1)/p^e}
\end{array}
\end{equation}
Set $\phi$ to be the map that sends $u_1^{1/p^{e}} r^{b/p^e}$ to $1$ and all other basis elements to zero.
Now consider the following set of elements of $S^{1/p^e}$.
\begin{equation}
\label{eq:newbasis}
\begin{array}{c}
u_1^{1/p^e}, \ldots, u_m^{1/p^e}\\
u_1^{1/p^e} u^{1/p^e} s^{(n \text{ mod } p^e)/p^e}, \ldots, u_m^{1/p^e} u^{1/p^e} s^{(n \text{ mod } p^e)/p^e}\\
\dots\\
u_1^{1/p^e} u^{(p^e - 1)/p^e} s^{(n(p^e - 1) \text{ mod } p^e)/p^e}, \ldots, u_m^{1/p^e} u^{(p^e - 1)/p^e} s^{(n(p^e - 1) \text{ mod } p^e)/p^e}
\end{array}
\end{equation}
Here we have used $(a \text{ mod } b)$ to denote the unique integer $0
\leq t \leq b-1$ such that $b |(a-t)$.

Our immediate goal will be to show that \eqref{eq:newbasis} is a basis
for $S^{1/p^e}$ over $S$.  Note that, since $p$ is relatively prime to
$n$, the integers $(ni \text{ mod } p^e)$ as we vary $i$ are precisely
$0, \ldots, p^{e}-1$.
By Lemma~\ref{lem:bases}, it is sufficient to show for a fixed integer
$i$ that the images of the $u_j^{1/p^e} u^{i/p^e}$ in $l^{1/p^e}$ (as
$j$ varies) are a basis for $l^{1/p^e}$ over $l$.  But that is obvious
since the $u_j^{1/p^e}$ form such a basis (since $k \subseteq l$ is
separable and using Lemma~\ref{lem:genericextension}) and multiplication by $u^{i/p^e}$ induces a vector space
automorphism of $l^{1/p^e}$ over $l$.

We know that $\phi$ has a unique extension $\bar{\phi} \in
\Hom_L(L^{1/p^e}, L)$ (first extending generically, then using
Lemma~\ref{lem:genericextension}).  Let us verify that
$\bar{\phi}(S^{1/p^{e}}) \subseteq S$ using the basis \eqref{eq:newbasis}.
Clearly, $\bar{\phi}$ sends $u_1^{1/p^e} u^{b/p^e} s^{(p^e - 1)/p^e} =
u_1^{1/p^{e}} r^{b/p^e}$ to $1$, and it is easy to see  that all other
elements of \eqref{eq:newbasis} must be sent to zero by construction.
Indeed, each of these elements agrees after multiplying by an integer
power of $s$ with an element of \eqref{eq:oldbasis} that is killed by
$\phi$.  This completes the proof of the main statement.  For the last
statement, note that the map $\Phi \in \Hom_R(R^{1/p^e}, R)$ which
sends $u_1^{1/p^e} r^{(p^e - 1)/p^e}$ to 1 and all other
elements of \eqref{eq:oldbasis} to 0 is a $R^{1/p^e}$-module generator
by Lemma~\ref{lem:GenMapDescription}.  As $\phi(\blank) =
\Phi(r^{(p^{e}-1-b)/p^{e}} \cdot \blank)$, we have
\[
\Delta_{\phi} = \frac{1}{p^{e}-1} \divisor_{R}(r^{(p^{e}-1-b)}) \, \, .
\]
Thus, we compute
\begin{eqnarray*}
\pi^{*} \Delta_{\phi} &=& \frac{1}{p^{e}-1} \divisor_{S}(r^{(p^{e}-1-b)})
= \frac{1}{p^{e}-1} \divisor_{S}(u^{(p^{e}-1-b)}s^{n(p^{e}-1-b)}) \\
&=& \frac{1}{p^{e}-1} \Div_S(s^{(n-1)(p^e-1)}) = \Div_S(s^{n-1}) =
\Ram_{\pi}
\end{eqnarray*}
as desired.
\end{proof}

Using the above Lemma, we now easily obtain the following criterion for extending
\mbox{$p^{-e}$-linear} maps over tamely ramified extensions of DVR's.

\begin{corollary}
\label{cor:tamelyRamifiedLifting}
In the context of Proposition \ref{prop:explictMapLiftsToGenerator}, a map $\psi \in \Hom_R(R^{1/p^e}, R)$ extends to a map $\bar{\psi} \in \Hom_S(S^{1/p^e}, S)$ if and only if $\Delta_{\psi} \geq \Delta_{\phi}$ (or in other words, if and only if $\pi^* \Delta_{\psi} \geq \Ram_{\pi}$).  Furthermore in this case, $\Delta_{\bar{\psi}} = \pi^* \Delta_{\psi} - \Ram_{\pi}$.
\end{corollary}

\begin{proof}
 We always have $\psi(\blank) = \phi(x^{1/p^e} \cdot \blank)$ for some
 $x \in K$, and hence also the extension $\bar{\psi} \in
\Hom_L(L^{1/p^e}, L)$ to $L$ satisfies $\bar{\psi}(\blank) =
\bar{\phi}(x^{1/p^e} \cdot \blank)$.  Using
Lemma~\ref{lem:CharacterizationOfGeneratingMap} , since $\bar{\phi}$
generates $\Hom_S(S^{1/p^e}, S)$, we have $\bar{\phi}(S^{1/p^{e}})
\subseteq S$ if and only if $x^{1/p^e} \in S^{1/p^e} \cap K^{1/p^e} =
R^{1/p^e}$.  This is precisely the statement that $\Delta_{\psi} \geq
\Delta_{\phi}$, and the remaining conclusion follows using
$\Delta_{\bar{\phi}} = 0$ (as it is a generator) and $\pi^*
\Delta_{\phi} = \Ram_{\pi}$.
\begin{comment}
The statement that $\Delta_{\psi} \geq \Delta_{\phi}$ simply means
that there exists $x \in R$ such that $\psi(\blank) = \phi(x^{1/p^e}
\cdot \blank)$.  We always have $\psi(\blank) = \phi(x^{1/p^e} \times
\blank)$ for some $x \in K$ (we will show that $x \in R$).  But then
$\bar{\psi}(\blank) = \bar{\phi}(x^{1/p^e} \times \blank)$ and so
since $\bar{\phi}$ generates $\Hom_S(S^{1/p^e}, S)$, we see by Lemma
\ref{lem:CharacterizationOfGeneratingMap} that $x^{1/p^e} \in
S^{1/p^e} \cap K^{1/p^e} = R^{1/p^e}$ as desired.
\end{comment}
\end{proof}

Let us now complete the alternative elementary proof of the main theorem, in contrast to \cite[Theorem 5.7]{SchwedeTuckerTestIdealFiniteMaps}.

\begin{proof}[Proof of Main Theorem]
It is sufficient to study when $\psi \in \Hom_R(R^{1/p^e}, R)$ can be extended to $\Hom_S(S^{1/p^e}, S)$ where $R \subseteq S$ is a tamely ramified generically finite inclusion of  DVR's. In the case that $e > 0$ is such that the ramification index divides $p^e - 1$, this happens if and only if $\pi^*
\Delta_{\psi} \geq \Ram_{\pi}$ by Corollary
\ref{cor:tamelyRamifiedLifting} above.  More generally, by applying
this criterion to an appropriately large Frobenius iterate
\[
\psi^m := \psi \circ \psi^{1/p} \circ \dots \circ \psi^{1/p^{m-1}}
\]
of $\psi$
and using that $\Delta_{\psi^m} = \Delta_{\psi}$ by \cite[Theorem 3.11(e)]{SchwedeFAdjunction},
we will have
completed our proof
as soon as we have shown the following lemma.
\end{proof}

\begin{lemma}
\label{lem.CompositionExtendsIfAndOnlyIf}
Suppose that $R \subseteq S$ is a finite inclusion of $F$-finite
normal domains.  Then an element $\psi \in \Hom_R(R^{1/p^e}, R)$ extends to an element $\bar{\psi} \in \Hom_S(S^{1/p^e}, S)$ if and only if the composition $\psi^m \in \Hom_R(R^{1/p^{me}}, R)$ extends to a map in $\Hom_S(S^{1/p^{me}}, S)$.
\end{lemma}
\begin{proof}
The forward implication is obvious, so we prove only the reverse.
Without loss of generality we may assume that $R$ is a DVR and $S$ is
a semi-local regular ring of dimension one, and set $L= \Frac S$.
Suppose that $\psi^m$ extends to a map $\overline{\psi^{m}} \in
\Hom_S(S^{1/p^{me}}, S)$, and set $\Phi$ to be a $S^{1/p^{e}}$-module
generator of $\Hom_S(S^{1/p^e}, S)$.  We can always extend $\psi$
 to $\bar{\psi} \in \Hom_L(L^{1/p^e}, L)$ and write $\bar{\psi}(\blank) =
\Phi(l^{1/p^e} \cdot \blank)$ for some $l \in L$. We will show that $l
\in S$ which will complete the proof by Lemma \ref{lem:CharacterizationOfGeneratingMap}.  Note that
$\overline{\psi^{m}}(\blank) = (\bar{\psi})^m (\blank)$ is simply
$\Phi^m(l^{a/p^{me}} \cdot \blank)$ for some integer $a >
0$. Therefore $l^{a/p^{me}} \in S^{1/p^{me}}$, whence $l \in S$ since $S$ is normal.
\end{proof}

%Our next example has wild ramification in codimension 1 because an extension of residue fields is non-separable.
\section{An example with wild ramification}
\label{sec.WildRamification}
We conclude this paper with an example showing the explicit extensions criteria can sometimes be obtained even for wild ramification, although the arguments we used above will not suffice.  In particular, even though the proof of Proposition \ref{prop:explictMapLiftsToGenerator} will not work in the wildly ramified case, there still can be explicit element by element ways to extend maps.  We do not know how to construct them except in an ad-hoc manner, which we do in the following example.

\begin{example}[Explicit extension with wild ramification]
\label{ex:nottame2}
  Suppose $R = \F_{3}[x,y] \subseteq \F_{3}[x,y,z]/\langle z^{3} - xz
  - y^{2} \rangle = S$ and consider $\phi \in \Hom_{R}(R^{1/3},R)$.  Identify
  $\phi$ with its generic extension $\phi \: K^{1/3} \to K$ to $K =
  \Frac(R)$, and denote by $\bar\phi$ the unique extension of $\phi$
  to $L = \Frac(S)$.  Since $\bq
  = xS$ is the only ramified height one prime (and lies over $\bp = xR$), we will check whether
  $\bar\phi((S_{\bq})^{1/3}) \subseteq S_{\bq}$.

  %Note that $R_P \subseteq S_Q$ has wild ramification because the residue field extension is purely inseparable.

As $S_{\bq}$ is a DVR (with uniformizer $x$), it follows that $(S_{\bq})^{1/3}$ is a free
$S_{\bq}$-module.  In fact, one can explicitly check that $\{ x^{i/3}z^{j/3}
\}_{0 \leq i,j \leq 2}$ form a basis.
Since $x^{3}z = x^{2}(xz) = x^{2/3}z^{3} - x^{2}y^{2}$ in $S$, we have
$\bar\phi(xz^{1/3}) = \bar\phi(x^{2/3}z - x^{2/3}y^{2/3})$ or
$  \bar\phi(z^{1/3}) = \frac{1}{x}\left( z
    \phi(x^{2/3}) -\phi(x^{2/3}y^{2/3}) \right)$. Thus, $\bar\phi(z^{1/3}) \in S_{\bq}$ if and only if $ z
    \phi(x^{2/3}) -\phi(x^{2/3}y^{2/3})$ maps to zero in the residue
    field $k(\bq) = \F_{3}(y)[z]/\langle z^{3}-y^{2} \rangle$.  Since
    (the images of) $1,z,z^{2}$ form a basis for
    $k(\bq)$ over $k(\bp)$, we conclude that $\phi(x^{2/3})$ and
    $\phi(x^{2/3}y^{2/3})$ map to zero in $k(\bp)$.

Continuing in this manner leads to the formulae in Table
\ref{tab:nottameformulae}
below and we conclude similarly that
$\bar\phi((S_{\bq})^{1/3})
\subseteq S_{\bq}$ if and only if each of the elements $\{ \phi(x^{i/3}y^{j/3}) \}_{1
\leq i,j \leq 2}$ map to zero in $k(\bp)$.  As
$\{x^{i/3}y^{j/3}\}_{0\leq i,j \leq 2}$ form a free basis for
$(R_{\bp})^{1/3}$ over $R_{\bp}$, one can verify explicitly that these
conditions are
equivalent to the concise statement $\ord_{R_{\bp}}(\Delta_{\phi}) \geq 1$.
\end{example}

\begin{table}[h]
\begin{equation*}
  \begin{array}{@{\quad}c@{\quad}|@{\quad}c@{\quad}}
     \hline & \\[-1.7ex]
    u & \mbox{$\bar\phi(u)$} \\
    \hline & \\[-1.7ex]
    z^{1/3} & \frac{1}{x}\left( z
    \phi(x^{2/3}) -\phi(x^{2/3}y^{2/3}) \right) \\[3pt]
    x^{1/3}z^{1/3} & z \phi(1) - \phi(y^{2/3})\\[3pt]
    x^{2/3}z^{1/3} & z \phi(x^{1/3}) - \phi(x^{1/3}y^{2/3})\\[3pt]
    z^{2/3} & \frac{1}{x} \left( z^{2}\phi(x^{1/3}) + z
      \phi(x^{1/3}y^{2/3}) + y \phi(x^{1/3}y^{1/3}) \right)\\[3pt]
    x^{1/3}z^{2/3} & \frac{1}{x} \left( z^{2} \phi(x^{2/3}) + z
      \phi(x^{2/3}y^{2/3}) + y \phi(x^{1/3}y^{1/3}) \right)\\[3pt]
    x^{2/3}z^{2/3} & z^{2}\phi(1)+z\phi(y^{2/3}) + y\phi(y^{1/3})\\[3pt]
\hline
  \end{array}
\end{equation*}
\caption{Extension formulae in Example \ref{ex:nottame2}.}
\label{tab:nottameformulae}
\end{table}

\bibliographystyle{skalpha}
\bibliography{CommonBib}  %%%-or upload your own bib file

\end{document}